\documentclass[12pt]{amsart}
\usepackage{amsmath,amssymb,amsbsy,amsfonts,amsthm,latexsym,
            amsopn,amstext,amsxtra,euscript,amscd,amsthm,graphicx,comment}
\usepackage[left=3.5cm,right=3.5cm,top=2cm,bottom=2cm,includeheadfoot]{geometry}
\usepackage[enableskew]{youngtab}

\newtheorem{lem}{Lemma}
\newtheorem{lemma}[lem]{Lemma}

\newtheorem{thm}{Theorem}
\newtheorem{theorem}[thm]{Theorem}

\newtheorem{cor}{Corollary}
\newtheorem{corollary}[cor]{Corollary}

\newtheorem*{conjecture}{Conjecture}

\theoremstyle{definition}

\newtheorem*{Acknowledgements}{Acknowledgements}

\theoremstyle{remark}

\def\\{\cr}
\def\({\left(}
\def\){\right)}
\def\<{\langle}
\def\>{\rangle}

\def\func#1{\mathop{\rm #1}}%

\begin{document}
\title{On the Non-vanishing of the D'Arcais Polynomials}
\author{Bernhard Heim}
\address{Department of Mathematics and Computer Science\\Division of Mathematics\\University of Cologne\\ Weyertal 86--90 \\ 50931 Cologne \\Germany}
\email{bheim@uni-koeln.de}
\author{Markus Neuhauser}
\address{Kutaisi International University, 5/7, Youth Avenue,  Kutaisi, 4600 Georgia}
\email{markus.neuhauser@kiu.edu.ge}
\address{Lehrstuhl f\"{u}r Geometrie und Analysis, RWTH Aachen University, 52056 Aachen, Germany}


\date{\today}
\begin{abstract}
In this paper we invest in the non-vanishing of the Fourier coefficients of powers of the Dedekind eta function. This is reflected in non-vanishing 
properties of the D'Arcais polynomials. We generalize and improve results of
Heim--Luca--Neuhauser and \.{Z}mija. We apply methods from algebraic number theory.
\end{abstract}
\maketitle
\section{Introduction}
In this paper, we study the vanishing properties of the $q$-expansion of $r$th powers of the
Dedekind $\eta$-function. Euler and Jacobi \cite{On03,Koe11} studied the odd cases $r=1$ and $r=3$ via
explicit formulas involving pentagonal and triangular numbers.
Serre \cite{Se85} considered the even case and proved that the sequence of coefficients is
lacunary if and only if $r \in S_{\func{even}}:=\{0,2,4,6,8,10, 14,26\}$. Lehmer \cite{Leh47} conjectured that the coefficients for $r=24$, which involves the discriminant function, called Ramanujan tau-function, never vanish. We also refer to Ono's speculation \cite{On95}
for $r=12$.
In general not much is known for integer powers.
In this paper we allow complex powers and consider all of them simultaneously, which leads to the study of D'Arcais polynomials $P_n^{\sigma}(x)$ \cite{HLN19, HNT20}.
Polynomization became recently a very active research field \cite{Li23}. 

The D'Arcais polynomials 
$P_n^{\sigma}(x)$ \cite{DA13, Co74, HN20} dictate the properties of
the coefficients of the powers of the Dedekind $\eta$-function \cite{Ne55,Se85,HNT20}.
The Dedekind $\eta$-function \cite{On03} is a modular form of weight $1/2$ and defined 
on the complex upper half-plane $\mathbb{H}$. Let $\tau \in \mathbb{H}$ and $q:= e^{2 \pi \,i \,\tau}$. Then
$$\eta(\tau):= q^{\frac{1}{24}} \, \prod_{n=1}^{\infty} \left( 1 - q^n \right).$$

One of the crucial facts is that the $n$th coefficients of the $q$-expansion of
$z$th powers of the infinite product  $\prod_{n=1}^{\infty} \left( 1 -q^n \right)$ are polynomial in $z$ of degree $n$. These are the D'Arcais polynomials. In combinatorics they are called
Nekrasov--Okounkov polynomials \cite{NO06, We06, Ha10}. 

Let $\sigma(n):= \sum_{d \mid
n} d$ and $z \in \mathbb{C}$. Then
\begin{equation*}\label{Einstieg}
\sum_{n=0}^{\infty} P_n^{\sigma}(z) \, q^n := \prod_{n=1}^{\infty} 
\left( 1 - q^n\right)^{-z} = 
\exp \left( z \sum_{n=1}^{\infty} \sigma(n) \, \frac{q^n}{n}\right).
\end{equation*}
Nekrasov and Okounkov \cite{NO06, We06,Ha10} discovered a new type of hook length formula
derived from
random partitions and the Seiberg--Witten theory.

Let $\lambda$ be a partition of $n$ denoted by $\lambda \vdash n$ with weight $|\lambda|=n$
and 
$\mathcal{H}(\lambda)$ the multiset of hook lengths associated
with $\lambda$. Let $\mathcal{P}$ be
the set of all partitions.
Then the Nekrasov--Okounkov hook length formula is given by
\begin{equation} \label{ON}
\sum_{ \lambda \in \mathcal{P}} q^{|\lambda|} \prod_{ h \in \mathcal{H}(\lambda)} \left(  1 - \frac{z}{h^2} \right) =   \prod_{m=1}^{\infty} \left( 1 - q^m \right)^{z-1}.
\end{equation}
Note that $A_n^{\sigma}(x):= n! \, P_n^{\sigma}(x) \in \mathbb{N}_0[x]$ are monic.
Therefore, the zeros are algebraic integers. For example,
this implies that the coefficients of
$\prod_{n=1}^{\infty} \left( 1 - q^n \right)^{\frac{1}{2}}$ 
are rational and non-vanishing.

By utilizing results from representation theory of simple complex Lie algebras, Kostant \cite{Ko04} proved that $P_n^{\sigma}\left(-(m^2-1)\right)$ does never vanish for $m \geq n$. Han deduced from (\ref{ON}) that for real $x$, we have already $P_n^{\sigma}(x) \neq 0$
if $\vert x \vert \geq n^2-1$.
Recently, these results have been significantly improved \cite{HN23}.
Let $c:= 9.7226$ and $x \in \mathbb{C}$. Then $P_n^{\sigma}(x)\neq 0$
for all $\vert x \vert > c \, (n-1)$. It is known that $c$ can not be smaller than $9.72245$. There is a conjecture \cite{HN19} on the real part of the zeros of $P_n^{\sigma}(x)$. Recall that
a polynomial is called Hurwitz polynomial or stable polynomial 
if the real part of its zeros is negative. 

\begin{conjecture}[Heim, Neuhauser \cite{HN19}]
Let $n \geq 1$. The D'Arcais polynomials $P_n^{\sigma}(x)$ divided by $x$ are Hurwitz polynomials and the zeros are simple.
\end{conjecture}

In particular it would be interesting to know if there are no zeros on the imaginary axes. 
As a first result in this direction, the following is known.
\begin{theorem}[Heim, Luca, Neuhauser \cite{HLN19}] 
\label{luca}
Let $m \geq 3$ and let $\zeta_m$ be a primitive root of unity. Then
for all $n \in \mathbb{N}$: 
\begin{equation*}
P_n^{\sigma}(\zeta_m)\neq 0.
\end{equation*}
In particular
$P_n^{\sigma}(\pm i ) \neq 0$.
\end{theorem}
Recently, \.{Z}mija \cite{Zm23} in his doctoral thesis
developed a remarkable generalization of Theorem \ref{luca}.
Let $g: \mathbb{N} \longrightarrow \mathbb{Z}$ with $g(1)=1$ be a normalized arithmetic function. 

Define
\begin{equation*}\label{general}
P_n^g(x) := \frac{x}{n} \sum_{k=1}^{n} g(k) \, P_{n-k}^g(x), \quad n \geq 1,
\end{equation*}
with initial value $P_0^g(x)=1$ (\cite{HNT20, Zm23}).

\begin{theorem}[\.{Z}mija \cite{Zm23}] 
\label{zmija}
Let $g$ be a normalized $\mathbb{Z}$-valued arithmetic function
and $P_n^g(x)$ integer-valued polynomials for all $n \in \mathbb{N}$.
Let $m \geq 3$ and $\zeta_m$ be a primitive root of unity. Let $A_n^g(x):= n! \, P_n^g(x)$. Assume
\begin{enumerate}
\item  modulo $5$: none of the polynomials $A_3^g(x)$ and $A_4^g(x)$ is divisible by a monic irreducible polynomial of degree $2$ over $\mathbb{F}_5$.
\item  modulo $7$: none of the polynomials $A_r^g(x)$ for $2 \leq r \leq 6$ is divisible by a monic irreducible polynomial of degree $4$ over $\mathbb{F}_7$.
\item
modulo $11$: none of the polynomials $A_r^g(x)$ for $2 \leq r \leq 10$ is divisible by a monic irreducible polynomial over $\mathbb{F}_{11}$ that divides
$x^{11^6 -1} -1$ and does not divide $x^{11^d-1}-1$ for $1 \leq d \leq 10$, $d \neq 6$.
\end{enumerate}
Then $P_n^g(\zeta_m) \neq 0$ for all $m$th roots of unity $\zeta_m$ of order at least $3$. In particular,
\begin{equation*} \label{imagi}
P_n^{g}(\pm i ) \neq 0.
\end{equation*}
\end{theorem}
The proofs of Theorem \ref{luca} and Theorem \ref{zmija}
are based on a careful analysis of $A_n^g(x) \pmod{p}$ for all prime numbers $p$ and an analytic argument based on properties of the Chebyshev function
obtained by Rosser and Schoenfeld \cite{RS75} and the utilization of 
the computer algebra system Mathematica.

In this paper we provide a new proof of the theorems and show
that the non-vanishing of $P_n^g(x)$ at values related to roots of unities in essence depend on the decomposition
$A_n^g(x) \pmod{2}$ and $A_n^g(x) \pmod{3}$ in $\mathbb{F}_2[x]$ and
$\mathbb{F}_3[x]$. Another ingredient of our proofs will be
the prime ideal decomposition of $p \mathcal{O}
_{K}$ in the cyclotomic field
$K=\mathbb{Q}(\zeta_m)$, where $\mathcal{O}_K$ is the ring of integers.
Further, we introduce a
certain algebraic integer $\alpha
$ with $K=\mathbb{Q}\left( \alpha
\right) $
and $p=2$ or $p=3$ does not divide the index $
\left[ \mathcal{O}_K: \mathbb{Z}
\left[ \alpha \right] \right] $.
Finally, we utilize the Dedekind--Kummer Theorem to obtain our results.

First we start with the improvement of Theorem \ref{luca} and Theorem \ref{zmija}. 

\begin{theorem}\label{theorem:main}
\label{heimneuhauser}Let $g$ be a normalized $\mathbb{Z}$-valued arithmetic function. Let $g(3)\equiv 0,1 \pmod{3}$.
Let $m \geq 3$ and $\zeta_m$ be an $m$th primitive root of unity. 
Then for all $n \in \mathbb{N}$: $P_n^{g}(\zeta_m)\neq 0$. 
In particular,
$P_n^{g}(\pm i ) \neq 0$.
\end{theorem}
Note that $\sigma(3) \equiv 1 \pmod{3}$. Therefore $P_n^g(\zeta_m)\neq 0$ for all $n \in \mathbb{N}$ and
$m \geq 3$. Further, $P_n^{\sigma}\left( {\zeta_2}\right) =0$ if
and only if $n$ is not a (generalized) pentagonal number.
\newline
\newline
To state our next result we
introduce the following notation.
Let $\mathcal{Z}_n^g$ be the set of zeros of $P_n^g(x)$. The set of all zeros is a subset
of the set of algebraic integers
and is denoted by
\begin{equation*}
\mathcal{Z}^g = \bigcup_{n=1}^{\infty} \mathcal{Z}_n^g \subset \overline{\mathbb{Z}}.
\end{equation*}
\renewcommand{\theenumi}{\roman{enumi}}
\renewcommand{\labelenumi}{\theenumi)}
\begin{theorem} \label{theorem:main}
Let $g$ be a normalized integer-valued arithmetic function.
Let $m \geq 3$ and $\zeta_m$ be
an $m$th primitive root of unity. 
Let $K$ be the $m$th cyclotomic field and $\mathcal{O}_K$ the ring of
integers.
Then we have the following.
\begin{enumerate}
\item  Let a prime $p\neq 2$ exist
such that $p \mid m$. Then
\begin{equation*}
\left\{ \zeta_m + 2 \, \beta \, : \, \beta \in \mathcal{O}_K \right\} \, \cap \, \mathcal{Z}^g = \emptyset.
\end{equation*}
\item  Let $m \geq 1$ be not of the form $
\left\{2^a \, 3^{\ell} \, : \, a \in \{0,1\} \text{ and }\ell \in \mathbb{N}_0 \right\}$.
Let $g(3) \equiv 0,1 \pmod{3}$
Then
\begin{equation*}
\left\{ \zeta_m + 3 \, \beta  \, : \, \beta  \in \mathcal{O}_K \right\} \, 
\cap \, \mathcal{Z}^g = \emptyset.
\end{equation*}
\end{enumerate}
\end{theorem}
\begin{corollary}
Let $g$ be a normalized integer-valued arithmetic function. 
Let $g(3) \equiv 0,1 \pmod{3}$. Then we have 
for all $m \geq 3$ and $n \geq 1$ that
$P_n^g( \zeta_m + 6 \beta) \neq 0$ 
for all $\beta \in \mathbb{Z}[\zeta_m]$.
\end{corollary}
\section{Results from Algebraic Number Theory}
We recall basic
notation and results. For further details, we refer to
\cite{La94,Leu96,ME04, Ma10}.
\subsection{Dedekind--Kummer}
Let $K \supseteq \mathbb{Q}$ be a number field and $\mathcal{O}_K$ 
the ring of integers in $K$.
There always exists an element $\alpha$ with 
$K=\mathbb{Q}(\alpha)$ called  primitive element.
Let us assume that $\alpha \in \mathcal{O}_K$. Let 
$\mathcal{O}_{K, \alpha}:= \mathbb{Z}[\alpha]$, which is an order in $K$.
The index
\begin{equation*}
\kappa_{\alpha} := [ \mathcal{O}_K : \mathcal{O}_{K, \alpha}].
\end{equation*}
is finite. Let $g_{\alpha}(x):= \func{Min}_{\alpha}(x)\in \mathbb{Z}[x]$ be the minimal polynomial of $\alpha$. 
Let $\Delta_K$ be the discriminant of $K$.
The discriminant $\func{disc}(g_{\alpha})$ of the minimal polynomial satisfies
\begin{equation*}
\func{disc}(g_{\alpha}) = \kappa_{\alpha}^2 \,\, \Delta_K.
\end{equation*}
The norm $N(\frak{a})$ of an
ideal $\mathfrak{a} \neq \{0 \}$ is defined as the finite index
of $\mathfrak{a}$ in $\mathcal{O}_K$.
Let $\mathfrak{a}$ be an ideal in $\mathcal{O}_K$, different from $\{0\}$ and $\mathcal{O}_K$.
Then $\mathfrak{a}$ decomposes uniquely into a finite product of prime ideals, up to the order, 
since $\mathcal{O}_K$ is a Dedekind domain. Let 
\begin{equation}\label{decomposition}
\mathfrak{a} = \mathfrak{p}_1^{e_1} \cdot    \mathfrak{p}_2^{e_2}         
 \cdots
\mathfrak{p}_g^{e_g}.
\end{equation}
Then $e_k$ denotes the ramification index of $\mathfrak{p}_k$. Let
$f_k$ be the residue class degree or inertial degree of $\mathfrak{p}_k$ 
determined by the norm $N(\frak{p}_k) = p^{f_k}$.
Here $ \mathbb{Z} \cap \mathfrak{p}_k= p \mathbb{Z}$.
A prime $p$ is called ramified if at least one ramification index $e_k$ in the decomposition of
$\mathfrak{a} = p \mathcal{O}_K$ is larger than $1$. It is known that a prime $p$
is ramified if and only if $p \mid
\Delta_K$.

Rather than first specifying the field, we start with an algebraic integer 
$\alpha$. We define $K:=\mathbb{Q}(\alpha)$ and have 
\begin{equation*}
K \supset \mathcal{O}_K \supseteq \mathcal{O}_{K, \alpha}.
\end{equation*}
We are interested in the prime ideal decomposition of the principal ideals $p \mathcal{O}_K$ in $\mathcal{O}_K$ and the associated ramification indices and inertia degrees.
This is a consequence of the Dedekind--Kummer theorem.
\begin{theorem}[Dedekind--Kummer]\label{D-K}
Let $\alpha$ be an algebraic integer and the primitive element of an algebraic number field $K:=\mathbb{Q}(\alpha)$. Let $p$ be any prime with $p \nmid \kappa_{\alpha}$. Let $g_{\alpha}\left( x
\right) $ be the minimal polynom of $\alpha$. Then
\begin{equation*}
p \mathcal{O}_K = \prod_k \mathfrak{p}_k^{e_k} \,\, 
\Longleftrightarrow \,\, 
g_{\alpha}(x) \equiv \prod_k \left( g_{\alpha,k}(x) \right)^{e_k} \pmod{p},
\end{equation*}
where the polynomials $g_{\alpha,k}(x) \in \mathbb{F}_p[x]$ are irreducible. Up to the order we have
$\func{degree}_{\mathbb{F}_p} g_{\alpha,k}\left( x\right) = f_k$.
\end{theorem}
Next, we take a closer look at cyclotomic fields.

\subsection{Cyclotomic Fields}
Let $m \geq 1$. We denote by
$\zeta_m$
a primitive $m$th root of unity. Let $K_m := \mathbb{Q}(\zeta_m)$
be the $m$th cyclotomic field. Then $K=K_{m}$ has a power basis with $\mathcal{O}_K = \mathbb{Z}[\zeta_m]$.
Moreover, $K_m \, / \mathbb{Q}$ is a Galois extension of degree $\varphi(m)$, where $\varphi$ is the Euler totient function. Further, $p$ is ramified in $\mathcal{O}_K$ if and only if $p \mid m$.
Since we have a Galois extension, the prime ideal decomposition (\ref{decomposition}) simplifies:
\begin{equation*}
p \mathcal{O}_K = \prod_{k=1}^g \mathfrak{p}_k^e, \text{ where } f_k = f \text{ and } \varphi(m)=e \, f \, g.
\end{equation*}
The inertial degree $f$ of $p$ in $\mathcal{O}_K$ can be explicitly calculated.
Let $m_p$ be the largest divisor of $m$, which is coprime to $p$.
Then, it is well-known that $f$ is the smallest positive integer such that
\begin{equation*}\label{f}
p^f \equiv 1 \pmod{m_{p}}.
\end{equation*}
Let $R_p$ be the set of $m \geq 1$ such that the inertial degree $f$ of
$p \mathcal{O}_{K_m}$ is $1$ (note that $f$ is unique since cyclotomic fields are Galois extensions).
This leads to
\begin{equation*}
R_2  = \left\{ 2^{\ell} \, : \, \ell \in \mathbb{N}_0 \right\}
\text{ and }
R_3  =  \left\{2^a \, 3^{\ell} \, : \, a \in \{0,1\} \text{ and }\ell \in \mathbb{N}_0 \right\}.
\end{equation*}

\subsection{Proof of Theorem \ref{heimneuhauser}}
We first consider $A_n^g(x) \pmod{p}$. 
\begin{lemma}\label{A:decomposition}
Let $g$ be a normalized $\mathbb{Z}$-valued arithmetic function. Let $p$ be a prime number. Then we have
\begin{equation*}\label{r}
A_{\ell\, p + r}^g(x) \equiv A_r^g(x) \, \left( A_p^g \right)^{\ell} \pmod{p}, \,\, \text{ where } 
0 \leq r < p.
\end{equation*}
Further, let $\{P_n^g(x)\}_n$
be integer-valued polynomials.
Then 
\begin{equation*}
A_p^g(x) \equiv x\, (x-1) \ldots (x-p+1) \pmod{p}.
\end{equation*}
\end{lemma}
\begin{proof} For $g=\sigma$ the proof is given in \cite{HLN19}. 
Let the polynomials be integer-valued.
Then the proof is given by \.{Z}mija (\cite{Zm23}, Lemma 5). 
The general case is proven in the same way.
The basic ingredient is provided by
\begin{equation*}
A_n^g(x) = x \left(           \sum_{k=1}^n   \frac{(n-1)!}{(n-k)!} \,\, g(k) \, \, A_{n-k}^g(x)\right).
\end{equation*}
Then we reduce $\pmod{p}$ the following polynomials
$$ A_{\ell \, p +1}^g(x), A_{\ell \, p +2}^g(x),
\ldots, A_{(\ell +1) \, p}^g(x)$$
step by step to $A_{\ell \, p}^g(x)$. 
\end{proof}
We have $A_0^g(x)=1$, $A_1^g(x)=x$ and $A_2^g(x) = x \left( x + g(2)\right)$.
Therefore, $A_n^g(x) \pmod{2}$ decomposes into linear factors for all $n$.

Let $m= 2^{\ell}$, $ \ell \geq 2$. Then we study $A_n^g(x) \pmod{3}$, which is essentially $A_3^g(x) \pmod{3}$.
\begin{lemma} \label{A3}
Let $g$ be a normalized $\mathbb{Z}$-valued arithmetic function. Then
$A_3^g(x) \equiv x^2 \, \left( x + 3 g(2)\right) \pmod{2}$. Further,
\begin{equation*}
A_3^g(x) \equiv x \left( x^2 - g(3) \right) \pmod{3}.
\end{equation*}
Therefore, $A_3^g(x) \pmod{3}$ decomposes into linear factors if
and only if $g(3) \equiv 0,1 \pmod{3}$.
\end{lemma}
\begin{proof}
 Further,
\begin{equation*}
A_3^g(x) = x \left( x^2 + 3 g(2)x + 2 g(3) \right).
\end{equation*}
Then the solutions of $A_3^g(x)=0$ are $x_1=0$ and 
\begin{equation*}
x_{2/3} = \frac{-3 g(2) \pm \sqrt{9 g(2)^2 - 8 g(3)}}{2}.
\end{equation*}
Therefore, there $A_3^g(x)/x \pmod{3}$ is irreducible if and
only if $g(3)$ is not a quadratic residue $\pmod{3}$.
\end{proof}
Suppose $P_n^g(\zeta_m) =0$. Then $\func{Min}_{\zeta_m}$ divides $A_n^g(x)$.
But since $2^{\ell} \not\in R_3$, we have that $\func{Min}_{\zeta_m}(x) \pmod{3}$ does not
decompose into linear factors, which is a contradiction to the assumption that $P_n^g(\zeta_m) =0$.

Now, let $m \neq 2^{\ell}$ and $m \geq 3$.
Then $A_n^g(x) \pmod{2}$ decomposes into linear factors.
Suppose $P_n^g(\zeta_m) =0$. Then $\func{Min}_{\zeta_m}$ divides $A_n^g(x)$.
But since $m \not\in R_2$, the same argument as before leads to a contradiction.
Therefore, Theorem \ref{heimneuhauser} is proven.
\subsection{Proof of Theorem \ref{theorem:main}}
Let $m \geq 3$ and $K=\mathbb{Q}(\zeta_m) \supset \mathcal{O}_K$.
The following result gives us control over the index 
$\kappa_{\alpha}=[\mathcal{O}_{K}: \mathbb{Z}[\alpha]]$
of some 
algebraic integer $\alpha$ with $K=\mathbb{Q}(\alpha)$.
\begin{lemma}\label{technical}
Let $\zeta_m$ be a
primitive root of unity for $m \geq 1$. Let $p$ be a prime number and $\mu$ an integer, such that $p \mid \mu $. Let 
$ K=\mathbb{Q}[\zeta_m]$ 
and $\mathcal{O}_K$ the ring of integers.
Let
\begin{equation*}
\alpha := \alpha_{\beta} = \zeta_m + \mu \, \beta, \text{ where } \beta \in \mathcal{O}_K.
\end{equation*}
Let $\mathcal{O}_{K, \alpha}:= \mathbb{Z}[\alpha]$ be the order associated to $\alpha$.
Then $p$ is coprime to the index 
$$\kappa_{\alpha}= \left[ \mathcal{O}_K : \mathcal{O}_{K, \alpha}\right].$$
\end{lemma}
\begin{proof}
We consider the exact sequence
\begin{equation*}\label{short}
 0 \rightarrow \mathcal{O}_{K,\alpha} \hookrightarrow \mathcal{O}_K \twoheadrightarrow 
 \mathcal{O}_K \, / \, \mathcal{O}_{K,\alpha}
 \rightarrow 0.
\end{equation*}
We apply the right exact functor 
$\otimes_{\mathbb{Z}} \mathbb{Z}/p \,  \mathbb{Z}$.
Let $M$ be a $\mathbb{Z}$ modul, then $$M \otimes_{\mathbb{Z}} \mathbb{Z}/p \,  \mathbb{Z}
\simeq M  /  pM$$ (\cite{JS14}). 
Since $\zeta_m \equiv \alpha \pmod{p \, \mathcal{O}_K}$, we obtain
\begin{equation*}
\mathcal{O}_{K, \alpha}  \big/  p \, \mathcal{O}_{K, \alpha} \, \simeq \, 
\mathcal{O}_{K}  \big/ \, p  \mathcal{O}_{K}.
\end{equation*}
Therefore,
\begin{equation*}
p \, \mathcal{O}_{K}  \big/  \mathcal{O}_{K, \alpha} \simeq \,  \mathcal{O}_{K}  \big/  \mathcal{O}_{K, \alpha}.
\end{equation*}
Since $\mathcal{O}_{K}  \big/  \mathcal{O}_{K, \alpha}$ is a 
finite abelian group the claim of the lemma follows (\cite{Leu96}, Chapter $4$).
\end{proof}
\begin{enumerate}
\item
Let $\alpha= \zeta_m + 2 \, \beta$ with $\beta \in \mathcal{O}_K$. Then $K=\mathbb{Q}[\alpha]$.
We deduce that $$ 2 \nmid \kappa_{\alpha} = 
[ \mathcal{O}_K : \mathcal{O}_{K,\alpha}]$$
from Lemma \ref{technical}.
Therefore, we can apply the Dedekind--Kummer Theorem \ref{D-K}. We have that
$\func{Min}_{\alpha}(x) \pmod{2}$ decomposes into linear factors 
$\Longleftrightarrow$ $m =2^{\ell
}$, $\ell \geq 2$.
On the other hand we have from Lemma \ref{A:decomposition} and $A_{1}^{g}\left( x\right) =x$, $A_{2}^{g}\left( x\right) = x(x+g(2))$
that
\begin{equation}
A_n^g(x) \pmod{2}\text{ decomposes into linear factors.}
\label{stern}
\end{equation}
Suppose $P_n^g(\zeta_m)=0$ for $n \geq 1$ and $m \neq 2^{\ell}$, $
\ell \geq 2$. Then
$\func{Min}_{\alpha}(x) \pmod{2}$ has to divide $A_n^g(x) \pmod{2}$. 
But this contradicts (\ref{stern}). 
\item  Due to $g(3) \equiv 0,1 \pmod{3}$, we deduce from 
Lemma \ref{A:decomposition} and Lemma \ref{A3} that $A_n^g(x) \pmod{3}$ decomposes into linear factors.
With the same reasoning as in i) and the assumption on $m$  we obtain the claim.
\end{enumerate}

\section{Open challenges and further study}
In the following we consider the sequence of D'Arcais polynomials $\{P_n^{\sigma}(x)\}_n$. Some of the questions can be also be asked for a more general type of polynomials $P_n^g(x)$, where $g$ is a non-negative normalized integer-valued arithmetic function. For example, let $g(n)=n^d$, $(d \in \mathbb{N}_0)$ then $d=1$ is related to Laguerre polynomials \cite{HNT20}.
\subsection{Imaginary axis}
We believe that the D’Arcais polynomials are Hurwitz polynomials. 
The zero at $x=0$ is irrelevant. It would be of interest to show that 
$x=0$ is the only zero on the imaginary axis, in particular that 
\[
f_n(t) := P_n^{\sigma}(it)
\]
is non-vanishing for all $t \in \mathbb{Z}\setminus\{0\}$. 
Theorem~\ref{theorem:main} with $m=4$ and $p=3$ implies that 
\[
f(3k \pm 1) \neq 0 \quad \text{for all } k \in \mathbb{Z}.
\] 
Moreover, an analysis of $A_n^{\sigma}(x) \pmod{7}$ yields 
\[
P_n^{\sigma}(\pm 3i) \neq 0.
\]

\subsection{Results of \.{Z}mija}
In this paper we generalized the results of Theorem \ref{luca} and Theorem
\ref{zmija} with the assumption that $g(3)\not\equiv 2 \pmod{3}$.
It would be interesting to clarify how this is compatible or related with 
the assumptions of 
Theorem \ref{zmija}. Note that we do not assume that $P_n^g(x)$ is integer-valued.
\subsection{Vanishing properties on the unit circle}
Let $g$ be an integer-valued normalized arithmetic function 
and 
$g(3) \not\equiv 2 \pmod{3}$. Let 
$$\mathbb{U}:= \{ z \in \mathbb{C}\, : \, \vert z \vert =1\}.$$
It follows from Theorem \ref{theorem:main}
that $P_n^g(x)$ is non-vanishing
at all primitive roots of unity $\zeta_m$ for $m \geq 3$.
Since $A_n^g(x) \in \mathbb{Z}[x]$ and normalized, zeros have to
be algebraic integers. Let $g=\sigma $.
Then $\zeta_1$ is not a zero and
$\zeta_2$ is a zero if and only if $n$ is not of the form $k (3k-1)/2$ for $k \in \mathbb{Z}$ (Euler's famous pentagonal theorem).
We expect that if $P_n^{\sigma}(\alpha)=0$ with $\alpha \in \mathbb{C}$ and $\lvert \alpha \rvert = 1$, then necessarily $\alpha = -1$.

\subsection{Dedeking-Kummer approach}
Find suitable algebraic integers $\alpha$, as for example $\zeta_m$ $(m \geq 3)$ with $P_n^{\sigma}(\alpha) \neq 0$. And analyse the
prime ideal decomposition of $p \mathcal{O}_K$ for $p=2,3$ and $5$ and
$K=\mathbb{Q}[\alpha]$ and the irreducible factors of the minimal polnomial $\pmod{p}$.

Vary $\alpha$ to $\beta$ such that still the same number field is generated
and that the index $[\mathcal{O}_K: \mathbb{Z}[\beta]]$ can be controlled.
\subsection{Stretching the unit circle}
Let $r \in \mathbb{N}$ or more generally let $r$ be an algebraic integer. 
Let
$$\mathbb{U}_r:= \{ r \, z \in \mathbb{C}\, : \, \vert z \vert =1\}.$$
Determine the $\alpha \in \mathbb{U}_r$ with $P_n^{\sigma} (\alpha) \neq 0$ for all $n \in \mathbb{N}$ (or almost all $n$).

\begin{Acknowledgements}
The authors thank Christian Kaiser for helpful discussions on variations of the Dedekind–Kummer theorem, and Johann Stumpenhusen for valuable suggestions.
\end{Acknowledgements}

\end{document}